\documentclass{birkjour}
\usepackage{amsmath}
\usepackage{amssymb}
\usepackage{amsthm}

\usepackage{verbatim}

\usepackage[utf8]{inputenc}

\usepackage{graphicx}
\usepackage{color}
\usepackage{enumerate}

\usepackage{esint}
\usepackage[all]{xy}
\usepackage{framed}
\usepackage{bbm}
\usepackage{subcaption}

\usepackage[super]{nth}

\usepackage{bm}

\theoremstyle{plain}
\newtheorem{theorem}{Theorem}[section]
\newtheorem{corollary}[theorem]{Corollary}
\newtheorem{lemma}{Lemma}[section]

\newtheorem{proposition}{Proposition}[section]

\theoremstyle{definition}
\newtheorem{definition}{Definition}[section]
\newtheorem{remark}{Remark}[section]

\numberwithin{equation}{section}

\newcommand{\bfd}{\mathbf{d}}

\newcommand{\bfx}{\mathbf{x}}

\newcommand{\bff}{\mathbf{f}}

\newcommand{\bfp}{\mathbf{p}}

\newcommand{\Hu}{\mathcal{H}u}
\newcommand{\Hw}{\mathcal{H}w}

\newcommand{\hu}{\widehat{\nabla u}}
\newcommand{\hv}{\widehat{\nabla v}}

\newcommand{\dd}{\mathrm{\,d}}

\newcommand{\hp}{\hat{\bfp}}

\newcommand{\mR}{\mathbb{R}}

\DeclareMathOperator{\sgn}{sgn}

\DeclareMathOperator{\dist}{dist}

\begin{document}

\author[\normalcolor K. K. Brustad]{Karl K. Brustad. \today}
\title{Infinity-harmonic functions in the plane:\\Regularity by injectivity}
\email{karl.brustad@momentumtechnologies.com}
\address{Momentum Technologies\\ Ranheimsveien 9\\ 7044 Trondheim\\ Norway}

\begin{abstract}
	It has been a long standing conjecture that the $ \infty $-harmonic functions in the plane have a 1/3-Hölder continuous gradient. It \emph{is} known that solutions are $ C^1 $ and that the gradient is locally $ \alpha $-Hölder, but $ \alpha $ comes without any positive lower bound. Aronsson's solution $ x^{4/3} - y^{4/3} $ shows that no better general regularity is possible.
	
	In the plane there is also a connection between the $ \infty $-Laplace equation and the one-dimensional heat equation, observed already by Aronsson himself. I shall show that this link can be accessed under a certain injectivity condition on the gradient, and that the caloric structure then is enough to prove the 1/3-Hölder continuity. Of course, an injective gradient is by no means a \emph{necessary} condition, as seen by smooth solutions such as the planes and cones.
\end{abstract}

\maketitle


\section{Introduction}
A function $ u\colon\Omega\to\mR $ is said to be $ \infty $-harmonic in $ \Omega\subseteq\mR^n $ if
\[ \Delta_\infty u = 0 \]
in the viscosity sense.
The $ \infty $-Laplace operator is $ \Delta_\infty u := \nabla u\Hu\nabla u^\top $ 
where $ \nabla u = [u_{x_1}, \dots, u_{x_n}]\colon\Omega\to\mR^n $ is the gradient and where $ \Hu = \left[u_{x_i, x_j}\right]_{i,j=1}^n\colon\Omega\to\mR^{n\times n} $ is the Hessian matrix of $ u $. In the plane the equation reads
\[ u_x^2u_{xx} + 2u_xu_yu_{xy} + u_y^2u_{yy} = 0. \]

\begin{theorem}
	Suppose $ u $ is $ \infty $-harmonic in an open $ \Omega\subseteq\mR^2 $ with $ \nabla u\neq 0 $.
	
	If
	\begin{equation}\label{eq: reverse lipschitz}
		|\nabla u(x) - \nabla u(y)| \geq C_1|x - y|,\qquad C_1>0,\;x,y\in\Omega,
	\end{equation}
	then, locally in $ \Omega $, there is a constant $ C_2>0 $ such that
	\begin{equation}\label{eq: hølder}
	|\nabla u(x) - \nabla u(y)| \leq C_2|x - y|^{1/3}.
	\end{equation}
\end{theorem}

The condition \eqref{eq: reverse lipschitz} -- sometimes called \emph{metric injectivity} -- clearly implies that the gradient $ \nabla u\colon\Omega\to \mR^2 $ is 1-1 onto $ B := \nabla u(\Omega) $, with a $ C_1^{-1} $-Lipschitz inverse $ \bff\colon B\to\Omega $.

\begin{remark}
	It is an interesting question whether the caloric realm can be accessed also with \eqref{eq: reverse lipschitz} weakened to a mere injectivity of the gradient. I think it probably does, but the proof will require better versions of Lemma \ref{lem: existence of potential} and Proposition \ref{prop: h is caloric}.
\end{remark}

\begin{remark}
	This paper is limited to the situation $ \nabla u\neq 0 $. It is perhaps possible to also include critical points, but, as we shall see, it would then be necessary to analyse the behaviour of the associated caloric functions in the limit $ t\to \infty $.
\end{remark}

\section{Proof of Theorem}

\begin{lemma}\label{lem: existence of potential}
	A Lipschitz inverse $ \bff $ of a continuous gradient field $ \nabla u $ is again a gradient field. A potential $ w\in C^{1,1} $ is given by
	\begin{equation}\label{eq: w definition}
		w(p) := p^\top \bff(p) - u(\bff(p)).
	\end{equation}
\end{lemma}

Note that the result is immediate if $ \bff $ is differentiable with Jacobian matrix $ \nabla\bff $. Because then
\[ \nabla w(p) = \bff^\top(p) + p^\top\nabla\bff(p) - \nabla u(\bff(p))\nabla\bff(p) = \bff^\top(p), \]
since $ \nabla u(\bff(p)) = p^\top $.

\begin{proof}
	I show that $ w(p+h) - w(p) = h^\top\bff(p) + o(|h|) $ as $ \mR^2\ni h\to 0 $.
	\begin{align*}
		w(p+h) - w(p) - h^\top\bff(p)
		&= (p+h)^\top \bff(p+h) - u(\bff(p+h))\\
		&\quad{} - p^\top \bff(p) + u(\bff(p)) - h^\top\bff(p)\\
		&= (p+h)^\top\big[\bff(p+h) - \bff(p)\big]\\
		&\quad{} - \nabla u(\bff(p))\big[\bff(p+h) - \bff(p)\big] + o(|\bff(p+h) - \bff(p)|)\\
		&= h^\top\big[\bff(p+h) - \bff(p)\big] + o(|\bff(p+h) - \bff(p)|),
	\end{align*}
	which is $ o(|h|) $ since $ |\bff(p+h) - \bff(p)|\leq C_1^{-1}|h| $. Thus $ w $ is differentiable with Lipschitz gradient $ \nabla w = \bff^\top $.
\end{proof}

Next, I introduce the smooth (and locally one-to-one) change of variables\footnote{in some branch $ D $ of $ \bfp^{-1}(B) $.} $ \bfp\colon D\to B $,
\[ p = \bfp(\psi, t) := e^{-t}\begin{bmatrix}
\cos\psi\\
\sin\psi
\end{bmatrix}, \]
and define $ h\in C^{1,1}(D) $ as
\begin{equation}\label{eq: h definition}
	h(\psi, t) := w(\bfp(\psi, t)).
\end{equation}
It is also convenient to write
\begin{equation}\label{eq: x definition}
	x = \bfx(\psi,t) := \bff(\bfp(\psi, t)).
\end{equation}
Note that
\[ \bfp_\psi = \bfp_\perp\qquad\text{and}\qquad \bfp_{\psi\psi} = -\bfp = \bfp_t, \]
where, for vectors $ a\in\mR^2 $, the subscript $ \perp $ stands for a 90 degrees counter-clockwise rotation. That is, $ a_\perp := Qa $ with $ Q :=
\left[\begin{smallmatrix}
0 & -1\\
1 & 0
\end{smallmatrix}\right]. $
It follows that
\begin{equation}\label{eq: h_psi, h_t}
	h_\psi = \nabla w(\bfp)\bfp_\psi = \bfx^\top\bfp_\perp\qquad\text{and}\qquad h_t = \nabla w(\bfp)\bfp_t = -\bfx^\top\bfp.
\end{equation}
To be clear, $ \nabla u(\bfx) = e^{-t}[\cos\psi, \sin\psi] $. Thus $ e^{-t} $ is the length $ |\nabla u(\bfx)| $ and $ \psi $ is the direction $ \arg\nabla u(\bfx) $ of the gradient.

By using \eqref{eq: h_psi, h_t} the homomorphism $ \bfx\colon D\to\Omega $ can be expressed in terms of $ h $ as
\begin{equation}\label{eq: x in terms of h}
	\begin{aligned}
	\bfx
	&= I\bfx\\
	&= \left(\hp\hp^\top + \hp_\perp\hp_\perp^\top\right)\bfx,\qquad\qquad\qquad (\hat{a} := a/|a|),\\
	&= e^{t}\left(h_\psi\hp_\perp - h_t\hp\right)\\
	&= e^t\begin{bmatrix}
	-\sin\psi & -\cos\psi\\
	\cos\psi & -\sin\psi
	\end{bmatrix}\begin{bmatrix}
	h_\psi\\
	h_t
	\end{bmatrix}.
	\end{aligned}
\end{equation}
The first of the two key results needed to prove the Theorem is the following.
\begin{proposition}\label{prop: h is caloric}
	The function $ h\colon D\to\mR $ given by \eqref{eq: h definition}, \eqref{eq: w definition} is caloric. That is, $ h\in C^\infty(D) $ with
	\[ h_{\psi\psi}(\psi, t) = h_t(\psi, t). \]
\end{proposition}

To prove Proposition \ref{prop: h is caloric} a fundamental result on the integral lines, or the \emph{streamlines}, of $ \nabla u $ is employed.

\begin{lemma}[Existence of maximal ascending streamlines with non-decreasing speed]\label{lem:fund}
	Let $ \Omega\subseteq\mR^n $, $ n \geq 2 $, be open and assume that $ u\in C^1(\Omega) $ is $ \infty $-subharmonic in $ \Omega $. i.e., $ \Delta_\infty u\geq 0 $. Let $ V\subset\subset\Omega $ be open. Then, for every $ x\in V $ with $ \nabla u(x)\neq 0 $ there is a $ T>0 $ and a curve $ \gamma\in C^1([0,T], \overline{V}) $ so that
	\begin{enumerate}[(i)]
		\item $ \gamma(0) = x $
		\item $ \gamma(T)\in\partial V $
		\item $ \dist(x,\partial V)\leq T \leq \frac{u(\gamma(T)) - u(x)}{|\nabla u(x)|} $
		\item $ \gamma'(s) = +\hu^\top(\gamma(s))$
		\item $ s\mapsto |\nabla u(\gamma(s))| $ is non-decreasing
		\item $ s\mapsto u(\gamma(s)) $ is convex
	\end{enumerate}
\end{lemma}

The symmetric result for \emph{descending} streamlines is obtained by setting $ v = -u $:

\begin{lemma}[Existence of maximal descending streamlines with non-decreasing speed]\label{lem:fundsup}
	Let $ \Omega\subseteq\mR^n $, $ n \geq 2 $, be open and assume that $ v\in C^1(\Omega) $ is $ \infty $-superharmonic in $ \Omega $. i.e., $ \Delta_\infty v\leq 0 $. Let $ V\subset\subset\Omega $ be open. Then, for every $ x\in V $ with $ \nabla v(x)\neq 0 $ there is a $ T>0 $ and a curve $ \gamma\in C^1([0,T], \overline{V}) $ so that
	\begin{enumerate}[(i)]
		\item $ \gamma(0) = x $
		\item $ \gamma(T)\in\partial V $
		\item $ \dist(x,\partial V)\leq T \leq \frac{v(x) - v(\gamma(T))}{|\nabla v(x)|} $
		\item $ \gamma'(s) = -\hv^\top(\gamma(s))$
		\item $ s\mapsto |\nabla v(\gamma(s))| $ is non-decreasing
		\item $ s\mapsto v(\gamma(s)) $ is concave
	\end{enumerate}
\end{lemma}

By gluing an ascending streamline and a descending streamline that start at a point $ x_0\in\Omega $ the two lemmas in combination implies

\begin{corollary}\label{cor: streamline touch}
	Let $ u\in C^1(\Omega) $ be $ \infty $-harmonic in an open $ \Omega\subseteq\mR^n $, $ n \geq 2 $, and let $ x_0\in\Omega $ where $ \nabla u(x_0)\neq 0 $. Then there exists a streamline $ \gamma\in C^1((-T,T),\Omega) $, $ T>0 $, with $ |\gamma'(s)|\equiv 1 $ so that $ \gamma(0) = x_0 $ and
	\begin{equation}
		|\nabla u(\gamma(s))| \geq |\nabla u(x_0)|,\qquad -T<s<T.
	\end{equation}
\end{corollary}

Lemma \ref{lem:fund} is essentially the $ C^1 $ version of Proposition 6.2 in \cite{Crandall2008}. The details of the proof are covered in the appendix.

\begin{proof}[Proof of Proposition \ref{prop: h is caloric}]
	For each fixed $ t_0 $ the curve $ \psi\mapsto \bfx^{t_0}(\psi) := \bfx(\psi,t_0) $ is Lipschitz, and thus absolutely continuous and differentiable at almost every $ \psi $.
	
	Suppose that $ \bfx^{t_0} $ is differentiable at $ \psi_0 $, and write $ x_0 = \bfx(\psi_0, t_0)\in\Omega $. For $ \psi $ close to $ \psi_0 $ the straight curve $ \bfd(\psi) := (\psi, t_0) $ in $ D $ locally divides $ D $ into the lower half $ D_l := \{(\psi, t)\,:\,t\leq t_0\} $ and the upper half $ D_u := \{(\psi, t)\,:\,t\geq t_0\} $. Since $ \bfx = \bff\circ\bfp\colon D\to\Omega $ is a homomorphism the topology is preserved, and the curve $ \psi\mapsto (\bfx\circ\bfd)(\psi) = \bfx(\psi,t_0) $ thus locally divides $ \Omega $ into $ \Omega_l := \bfx(D_l) = \{x\,:\, |\nabla u(x)|\geq |\nabla u(x_0)|\} $ and $ \Omega_u := \bfx(D_u) = \{x\,:\, |\nabla u(x)|\leq |\nabla u(x_0)|\} $.
	
	By Corollary \ref{cor: streamline touch} a regular $ C^1 $ curve $ \gamma $ passes through $ x_0 $ that lays entirely in $ \Omega_l $. This means that, while the two curves $ \bfx^{t_0} $ and $ \gamma $ touch at $ x_0 $, \emph{they do not cross}. So, if $ \frac{\dd}{\dd\psi}\bfx^{t_0}(\psi_0)\neq 0 $ the tangent lines of the curves must be parallel at the touching point. i.e., $ \bfx_\psi(\psi_0, t_0) \parallel \gamma'(0) $. In any case, zero or not, and since $ \gamma'(0) = \hu^\top(x_0) = \hu^\top(\bfx(\psi_0, t_0)) = \hp(\psi_0, t_0) $, it follows that for each fixed $ t $, $ \psi\mapsto\bfx $ is differentiable with
	\begin{equation}
		0 = \bfx_\psi^\top Q\bfp = \bfx_\psi^\top \bfp_\perp, \qquad\text{at almost every $ \psi $.}
	\end{equation}
	Next, for $ \psi\mapsto h_\psi = \bfx^\top\bfp_\perp $ (see \eqref{eq: h_psi, h_t}) we then have
	\[ h_{\psi\psi} = \bfx^\top_\psi\bfp_\perp + \bfx^\top\frac{\partial}{\partial\psi}\bfp_\perp = 0 - \bfx^\top\bfp, \qquad\text{a.e.,}\]
	and, since $ \psi\mapsto h_\psi(\psi,t) $ is absolutely continuous,
	\[ h_\psi(\psi,t) - C(t) = \int_{\psi'}^\psi h_{\psi\psi}(\phi,t)\dd\phi = - \int_{\psi'}^\psi \bfx^\top(\phi,t)\bfp(\phi,t)\dd\phi\]
	in the sense of Lebesgue. But the integrand is continuous, so $ \psi\mapsto h_\psi $ is differentiable  \emph{everywhere} with $ h_{\psi\psi} = - \bfx^\top\bfp $. By \eqref{eq: h_psi, h_t} this equals $ h_t $.
	
\end{proof}

The $ \infty $-harmonic function $ u $ as a function of $ \psi, t $ is
\begin{equation}\label{eq: v definition}
	v(\psi,t) := u(\bff(\bfp(\psi,t))) = u(\bfx(\psi, t)).
\end{equation}
Rearranging $ w(p) = p^\top\bff(p) - u(\bff(p)) $ and using \eqref{eq: x definition}, \eqref{eq: h_psi, h_t}, gives
\begin{align*}
	v &= u(\bfx)\\
	  &= \bfp^\top\bfx - w(\bfp)\\
	  &= - h_t - h,
\end{align*}
which shows that $ v $ is smooth and caloric, as well.

The final key result towards the proof of the Theorem is to show that $ \nabla v := [v_\psi, v_t] \neq 0 $ in $ D $.

\begin{proposition}\label{prop: key 2}
	Under the assumptions of the Theorem, let $ v\in C^\infty(D) $ be defined as in \eqref{eq: v definition}. If $ v_\psi = 0 $ at $ (\psi_0, t_0) \in D $ then
	\[ v_t(\psi_0, t_0) \neq 0. \]
\end{proposition}

\begin{proof}
	First, there is no harm assuming $ u = 0 $ at $ x_0 := \bff(\bfp(\psi_0, t_0)) $, so that $ v(\psi_0, t_0) = 0 $ as well. In fact, we may also assume $ (\psi_0, t_0) = (0,0) $ and $ x_0 = 0\in\Omega $, since all this is obtained by just translating, scaling and rotating $ \Omega $ so that $ \nabla u(x_0) = \nabla u(0) = [1,0] $.
	
	From \eqref{eq: w definition}, \eqref{eq: h definition}, \eqref{eq: h_psi, h_t}, and by the assumption $ 0 = v_\psi = - h_{t\psi} - h_\psi $, we then have
	\begin{equation}\label{eq: h 4 moments are 0}
		h = h_\psi = h_t = h_{t\psi} = 0\qquad\text{at $ (\psi, t) = (0, 0) $.}
	\end{equation}
	
	The function $ v(\psi,0) $ is not identically zero, because if it was then $ 0 = -v = h + h_t = h + h_{\psi\psi} $, which solves to $ h(\psi,0) = A\cos\psi + B\sin\psi $. The boundary conditions \eqref{eq: h 4 moments are 0} then makes $ h(\psi,0) $ identically zero too and, by \eqref{eq: x in terms of h}, $ \bfx(\psi, 0)\equiv 0 $. That is not a homomorphism.
	
	For a fixed time, it is known that caloric functions in the plane are analytic in the spatial variable. We may therefore write
	\begin{equation}\label{eq: v Taylor series}
		v(\psi, 0) = \sum_{n=2}^{\infty}a_n \psi^n,\qquad\text{near $ \psi = 0 $,}
	\end{equation}
	where, by the argument above, at least some of the $ a_n $'s are non-zero. However, $ a_0 = v(0,0) = 0 $ and $ a_1 = v_\psi(0,0) = 0 $.
	
	Next I do a (two-dimensional) Taylor-approximation of $ v $ around $ (\psi, t) = (0,0) $, and gather the terms with the same coefficient $ a_n = \frac{\partial^n}{\partial\psi^n}v(0,0) /n! $ into polynomials $ r_n $. That is,
	\[ v(\psi, t) = \sum_{n=2}^N a_n r_n(\psi, t) + E_N(\psi, t). \]
	It can be shown that $ r_n $ is the $ n $-th order \emph{heat polynomial}
	\[ r_n(\psi, t) = \sum_{k=0}^{\lfloor n/2\rfloor} \frac{n!}{k!(n-2k)!}\psi^{n-2k}t^k,\] 
	which is caloric and \emph{parabolically $ n $-homogeneous},
	\begin{equation}
		r_n(\lambda\psi,\lambda^2t) = \lambda^n r_n(\psi, t),\qquad \lambda\in\mR.
	\end{equation}
	The first few are
	\[ r_0 = 1,\quad r_1 = \psi,\quad r_2 = \psi^2 + 2t,\quad r_3 = \psi^3 + 6\psi t,\quad r_4 = \psi^4 + 12\psi^2 t + 12t^2. \]
	The error term thus have the property $ E_N(\epsilon\psi, \epsilon^2t) = o(|\epsilon|^N)$ in bounded domains.
	
	An important fact about the heat polynomials is that they can be factored as
	\begin{equation}\label{eq: r_n product}
		r_n(\psi, t) = \frac{n!}{\lfloor n/2\rfloor!}\psi^{\sigma_n}\prod_{k=1}^{\lfloor n/2\rfloor} \left(t + b_{n,k}\psi^2\right),
	\end{equation}
	for some positive distinct numbers $ 0 < b_{n,1} < b_{n,2} < \cdots < b_{n,{\lfloor n/2\rfloor}} $.
	See e.g. Theorem 2.4 in \cite{BadgerJeznach2026}. Here, $ \sigma_n = \frac{1-(-1)^n}{2} $ is zero for $ n $ even, and 1 for $ n $ odd.
	
	Now, let $ a_N $, $ N\geq 2 $, be the first non-zero coefficient in the Taylor series \eqref{eq: v Taylor series}. Assume initially that $ N = 3 $. Then $ v(\psi, t) = a_3r_3(\psi, t) + E_3(\psi, t) $, which scales to
	\begin{equation}\label{eq: v approx to r_3}
		\frac{1}{a_3\epsilon^3}v(\epsilon\psi, \epsilon^2t) = \psi(\psi^2 + 6t) + o(1)\qquad\text{as $ \epsilon\to 0 $}.
	\end{equation}
	Notice that the null-levelset of $ r_3 $ is the $ t $-axis in union with a downward-opening parabola. The number of \emph{nodal domains}\footnote{the number of connected components of $ \{(\psi,t)\;|\; r_3(\psi, t) \neq 0\} $.} is thus four, and $ r_3 $ changes sign at least four times along any simple closed curve around the origin -- and \emph{exactly} four times for such curves that enters each nodal domain only once. By \eqref{eq: v approx to r_3}, the same is true for $ v $ on curves sufficiently close to $ (0,0)\in D $. For example, for $ \theta\mapsto v(\bfd(\theta)) $ on the ellipse
	\[ \bfd(\theta) = \begin{bmatrix}
	\psi(\theta)\\
	t(\theta)
	\end{bmatrix} = \begin{bmatrix}
	\epsilon\cos\theta\\
	\epsilon^2\sin\theta
	\end{bmatrix} \]
	provided, say, $ |o(1)|\leq 1 $.
	
	Since $ \bfx\colon D\to\Omega $ is a homomorphism, the curve $ \omega(\theta) := \bfx(\bfd(\theta)) $ is also simple and closed around $ x_0 = 0 $ in $ \Omega $, and
	\[ u(\omega(\theta)) = u(\bfx(\bfd(\theta))) = v(\bfd(\theta)) \]
	changes sign four times, as well. But this is a contradiction since $ \nabla u \neq 0 $, and the null-levelset of $ u $ passes trough $ x_0 $ and is locally a straight line.
	
	The situation gets no better if $ a_N $, $ N\geq 4 $, is the first non-zero coefficient in the Taylor series \eqref{eq: v Taylor series}: By \eqref{eq: r_n product} the number of nodal domains of $ r_N $ is $ 2\lceil N/2\rceil \geq 4 $, and
	\[ \frac{1}{a_N\epsilon^N}v(\epsilon\psi, \epsilon^2t) = r_N(\psi, t) + o(1)\qquad\text{as $ \epsilon\to 0 $} \]
	in the same way as in \eqref{eq: v approx to r_3}.
	
	It follows that none of $ a_3r_3, a_4r_4, a_5r_5, \dots $ can be the first non-zero term, and we have proved that
	\[ 0 \neq 2! a_2 = v_{\psi\psi} = v_t \]
	at $ (\psi_0, t_0) $.
\end{proof}

With this last result I am going to show that
\begin{equation}\label{eq: f growth lower bound}
	|\bff(p + h) - \bff(p)| \geq C|h|^3
\end{equation}
locally in $ B = \nabla u(\Omega) $, (this is where the "3" in \eqref{eq: hølder} comes from). We have $ \bff\in C^\infty(B,\Omega) $, and differentiating $ \bff(\bfp) = \bfx = e^{t}\left(h_\psi\hp_\perp - h_t\hp\right) $ with $ \psi $ and $ t $ gives
\begin{equation}\label{eq: x del psi}
	\begin{aligned}
	\nabla\bff(\bfp)\bfp_\perp = \bfx_\psi
	&= e^{t}\left(h_{\psi\psi}\hp_\perp - h_\psi\hp - h_{t\psi}\hp - h_{t}\hp_\perp\right)\\
	&= e^tv_\psi\hp
	\end{aligned}
\end{equation}
and
\begin{equation}\label{eq: x del t}
	\begin{aligned}
	-\nabla\bff(\bfp)\bfp = \bfx_t
	&= \bfx + e^{t}\left(h_{t\psi}\hp_\perp - h_{tt}\hp\right)\\
	&= e^t\left(v_t\hp - v_\psi\hp_\perp\right),
	\end{aligned}
\end{equation}
respectively.
From this it is seen that $ \nabla\bff = \Hw $ is never the zero-matrix (Proposition \ref{prop: key 2}), and that $ p_\perp^\top\nabla\bff(p)p_\perp = 0 $.\footnote{or $ p_1^2w_{p_2p_2} - 2p_1 p_2w_{p_1p_2} + p_2^2w_{p_1p_1} = 0 $ when written out in terms of the second-order partial derivatives of $ w $, with variables $ p = (p_1, p_2) = (u_x(x,y), u_y(x,y)) $.}

Since $ \hp = e^t\bfp $, multiplying with $ e^t $ and taking cross product yield 
\[ \det\Hw(\bfp) = \det\nabla\bff(\bfp) = - e^t\bfx_\psi\times e^t\bfx_t = -e^{2t}\bfx_\psi^\top Q\bfx_t = -e^{4t}v_\psi^2.\]
By the Inverse Function Theorem, the inverse $ \nabla u $ of $ \bff $ is thus smooth near points $ x = \bfx(\psi,t) $ where $ v_\psi(\psi,t)\neq 0 $.


We need to do a third order expansion of $ \bff $. This can be tedious, but by using matrix derivatives with the following notation the calculations can be kept relatively tidy and efficient.

\begin{definition}
	Let $ B\subseteq\mR^n $ be open and let $ A\colon B\to\mR^{N\times M} $ be smooth. Then for $ e\in\mR^M $,
	\[ \nabla_e A(p) := \nabla[A(p)e] \in \mR^{N\times n}.\]
	That is, $ \nabla_e A $ is the Jacobian matrix of the vector field $ p\mapsto A(p)e\in \mR^N $.
\end{definition}

When $ e\colon\Omega\to\mR^M $ itself is a smooth vector field, with Jacobian matrix $ \nabla e(p)\in\mR^{M\times n} $, the product rule is $ \nabla [Ae] = \nabla_e A + A\nabla e $. It can be shown that when $ A $ is a Hessian $ (N = M = n) $ then $ \nabla_e A \in \mR^{n\times n}$ is symmetric and $ \nabla_b Aa = \nabla_a Ab $ for all $ a,b\in\mR^n $. In this case we also write $ \nabla_{ab}A $ for the second-order derivative $ \nabla_b\nabla_a A = \nabla[\nabla_aAb] = \nabla[\nabla_bAa] = \nabla_{ba}A $. The second-order Taylor polynomial of $ A $ about $ p\in B $ can thus be written as $ A(p+e) \approx A(p) + \nabla_{e}A(p) + \nabla_{ee}A(p) $.

\begin{lemma}\label{lem: Hw higer derivatives}
	At all $ p\in B = \nabla u(\Omega) $ we have the following identities in $ \mR^2 $ and $ \mR^{2\times 2} $, respectively.
	\begin{enumerate}[I)]
		\item \[ \left(2Q^\top\Hw(p) + \nabla_{p_\perp}\Hw(p)\right)p_\perp  = 0\]
		\item \[ \frac{1}{2}\nabla_{p_\perp p_\perp}\Hw(p) = -Q^\top\nabla_{p_\perp}\Hw(p) - \nabla_{p_\perp}\Hw(p)Q - Q^\top\Hw(p)Q. \]
	\end{enumerate}
	Furthermore. If $ p = \bfp(\psi_0, t_0) $ where $ v_\psi(\psi_0, t_0) = 0 $, then
	\begin{enumerate}[A)]
		\item $ \Hw(p)p_\perp = 0 $
		\item $ \Hw(p)p = -e^{2t_0}v_t p  $
		\item $ \nabla_{p_\perp}\Hw(p)p_\perp = 0 $
		\item $ p_\perp^\top\nabla_{p_\perp p_\perp}\Hw(p)p_\perp = 2v_t $
	\end{enumerate}
\end{lemma}

\begin{proof}
	Starting with the PDE $ 0 = p_\perp^\top\Hw(p)p_\perp = p^\top Q^\top\Hw(p)Qp $, the gradient of the right-hand side is
	\[ 0 = p_\perp^\top\left(\nabla_{p_\perp}\Hw + \Hw Q\right) + p_\perp^\top\Hw Q. \]
	Collecting terms and transposing gives I).
	
	Before differentiating I) we multiply from the left with $ e\in\mR^2 $ to get the variable $ p_\perp $ out of $ \nabla_{p_\perp}\Hw $.
	\[ 0 = e^\top\left(2Q^\top\Hw + \nabla_{p_\perp}\Hw\right)p_\perp = \left(2e_\perp^\top\Hw + p_\perp^\top\nabla_{e}\Hw\right)p_\perp. \]
	The gradient is then
	\begin{align*}
		0 &= \left(2e_\perp^\top\Hw + p_\perp^\top\nabla_{e}\Hw\right)Q\\
		  &\quad{} + p_\perp^\top\left(2\nabla_{e_\perp}\Hw + \nabla_{p_\perp e}\Hw + \nabla_{e}\Hw Q\right)\\
		  &= e^\top\left(2Q^\top\Hw + \nabla_{p_\perp}\Hw\right)Q\\
		  &\quad{} + e^\top\left(2Q^\top\nabla_{p_\perp}\Hw + \nabla_{p_\perp p_\perp}\Hw + \nabla_{p_\perp}\Hw Q\right)\\
		  &= e^\top\left(2Q^\top\Hw Q + 2\nabla_{p_\perp}\Hw Q + 2Q^\top\nabla_{p_\perp}\Hw + \nabla_{p_\perp p_\perp}\Hw\right)
	\end{align*}
	which is II) since $ e $ is arbitrary.
	
	The claims A)-D) then follows readily from \eqref{eq: x del psi}, \eqref{eq: x del t}, and I), II).
\end{proof}

For the proof of the Theorem to work it is important that the Jacobian matrix $ \Hw $ of $ \bff $ is both singular and non-zero when $ v_\psi = 0 $. This is illustrated by the next lemma, which establish that the minimal change in $ \bff $ must be in the general heading of the singular direction $ \hat{p}_\perp $ of $ \Hw $.

\begin{lemma}
	Let $ p = \bfp(\psi_0, t_0) $ be a point where $ v_\psi(\psi_0, t_0) = 0 $, and let $ \epsilon_0>0 $ be small.
	Then there is a constant $ C>0 $ such that for all $ 0<\epsilon\leq \epsilon_0 $
	\[ \min_{|h|=\epsilon}|\bff(p+h) - \bff(p)| = \min_{\substack{|h|=\epsilon\\
	\left|\hat{h}^\top\hat{p}\right|\leq C\epsilon}}|\bff(p+h) - \bff(p)|. \]
\end{lemma}

\begin{proof}
	Set
	\[ C = \frac{1}{\alpha}\max_{\substack{|e| = 1\\
			|q-p|\leq\epsilon_0}}\left|\nabla_e\Hw(q)e\right|,\qquad \alpha = e^{2t_0}|v_t(\psi_0, t_0)| > 0. \]
	Since $ v_\psi = 0 $, the Jacobian matrix of $ \bff $ at $ p $ is $ \Hw(p) = \pm\alpha\hat{p}\hat{p}^\top $. Let $ 0 < \epsilon\leq\epsilon_0 $ and chose $ h $ such that $ |h|=\epsilon $, but with $ \left|\hat{h}^\top\hat{p}\right| > C\epsilon $. Then there is a $ p'\in B_{\epsilon_0}(p) $ such that
	\begin{align*}
		\left|\bff(p+h) - \bff(p)\right|
		&= \left|\Hw(p)h + \frac{1}{2}\nabla_h\Hw(p')h\right|\\
		&= \left|\pm\epsilon\alpha\hat{p}\hat{p}^\top\hat{h} + \frac{\epsilon^2}{2}\nabla_{\hat{h}}\Hw(p')\hat{h}\right|\\
		&\geq \epsilon\alpha|\hat{h}^\top\hat{p}| - \frac{\epsilon^2}{2}\left| \nabla_{\hat{h}}\Hw(p')\hat{h}\right|\\
		&> \epsilon^2\alpha C - \epsilon^2\alpha C/2\\
		&= \epsilon^2\alpha C/2.
	\end{align*}
	On the other hand,
	\[ \left|\bff(p+\epsilon\hat{p}_\perp) - \bff(p)\right|
	= \left|0 + \frac{\epsilon^2}{2}\nabla_{\hat{p}_\perp}\Hw(p'')\hat{p}_\perp\right| \leq \epsilon^2\alpha C/2,\]
	which shows that the minimum cannot be found outside the double cone $ \left|\hat{h}^\top\hat{p}\right| \leq C\epsilon $. The claim follows.
\end{proof}

Next, for $ h\in\mR^2 $ with $ |h| =\epsilon $ we can write
\[ h = \epsilon\hat{h} = \epsilon\left(\hat{h}^\top\hat{p}\cdot\hat{p} + \hat{h}^\top\hat{p}_\perp\cdot\hat{p}_\perp
\right) = \epsilon\left(h_1\hat{p} + h_2\hat{p}_\perp
\right). \]
If $ \left|\hat{h}^\top\hat{p}\right| \leq C\epsilon $ then $ h_1 := \hat{h}^\top\hat{p} $ is $ O(\epsilon) $ and $ h_2 := \hat{h}^\top\hat{p}_\perp\to 1 $ as $ \epsilon\to 0 $. The tensor $ \nabla_h\Hw(p) h $ is bi-linear in $ h $ and thus expands according to the binomial formula as
\[ \frac{1}{\epsilon^2}\nabla_h\Hw(p) h
=  h_1^2\nabla_{\hat{p}}\Hw(p)\hat{p} + 2h_1h_2\nabla_{\hat{p}}\Hw(p)\hat{p}_\perp + h_2^2\nabla_{\hat{p}_\perp}\Hw(p)\hat{p}_\perp, \]
which in general is $ \nabla_{\hat{p}_\perp}\Hw(p)\hat{p}_\perp + O(\epsilon) $ as $ \epsilon\to 0 $. But when $ p = \bfp(\psi_0, t_0) $ is a point where $ v_\psi(\psi_0, t_0) = 0 $, we have $ \nabla_{\hat{p}_\perp}\Hw(p)\hat{p}_\perp = 0 $ by Lemma \ref{lem: Hw higer derivatives} C), and multiplying on the left with $ \hat{p}_\perp $ gives
\begin{equation}\label{eq: Hw first order perp}
	\begin{aligned}
	\frac{1}{\epsilon^2}\hat{p}_\perp^\top\nabla_h\Hw(p) h
	&=  h_1^2\hat{p}_\perp^\top\nabla_{\hat{p}}\Hw(p)\hat{p} + 2h_1h_2\hat{p}_\perp^\top\nabla_{\hat{p}}\Hw(p)\hat{p}_\perp + h_2^2\hat{p}_\perp^\top\nabla_{\hat{p}_\perp}\Hw(p)\hat{p}_\perp\\
	&=  h_1^2\hat{p}_\perp^\top\nabla_{\hat{p}}\Hw(p)\hat{p} + 0 + 0\\
	&= O(\epsilon^2)
	\end{aligned}
\end{equation}
when $ \epsilon\to 0 $.\footnote{In the second term the symmetries $ \nabla_a\Hw^\top = \nabla_a\Hw $ and $ \nabla_a\Hw\, b = \nabla_b\Hw\, a $ are used.} 
Similarly, $ \nabla_{hh}\Hw(p) h $ is tri-linear in $ h $, and only the fourth and last term in the polynomial comes without $ \epsilon $ factors. That is,
\[ \frac{1}{\epsilon^3}\nabla_{hh}\Hw(p) h = \nabla_{\hat{p}_\perp\hat{p}_\perp}\Hw(p)\hat{p}_\perp + O(\epsilon). \]
When $ v_\psi = 0 $, we get
\begin{equation}\label{eq: Hw second order perp}
	\frac{1}{\epsilon^3}\hat{p}_\perp^\top\nabla_{hh}\Hw(p) h = 2\frac{v_t}{|p|^4} + O(\epsilon) = 2e^{4t_0}v_t + O(\epsilon)
\end{equation}
by Lemma \ref{lem: Hw higer derivatives} D).

Putting it all together, and using that $ |a| = \max_{|e|=1}e^\top a $, we get that at points $ p = \bfp(\psi_0, t_0) $ where $ v_\psi(\psi_0, t_0) = 0 $,
\begin{align*}
	\min_{|h|=\epsilon}&\frac{\left|\bff(p+h) - \bff(p)\right|}{|h|^3}\\
	&= \frac{1}{\epsilon^3} \min_{\substack{|h|=\epsilon\\
			\left|\hat{h}^\top\hat{p}\right|\leq C\epsilon}} |\bff(p+h) - \bff(p)|\\
	&\geq \frac{1}{\epsilon^3} \min_{\substack{|h|=\epsilon\\
			\left|\hat{h}^\top\hat{p}\right|\leq C\epsilon}} \sgn(v_t)\hat{p}_\perp^\top\left[ \Hw(p)h + \frac{1}{2!}\nabla_h\Hw(p)h + \frac{1}{3!}\nabla_{hh}\Hw(p) h + O(\epsilon^4)\right]\\
	&= \frac{1}{\epsilon^3} \left[ 0 + \frac{1}{2!}O(\epsilon^4) + \frac{\epsilon^3}{3!}\left(2e^{4t_0}|v_t| + O(\epsilon)\right) + O(\epsilon^4)\right]\\
	&= \frac{1}{3}e^{4t_0}|v_t| + O(\epsilon).
\end{align*}
Thus for every $ h $ small enough, we have, for, say $ C = \frac{1}{4}e^{4t_0}|v_t| > 0 $,
\[ \left|\bff(p+h) - \bff(p)\right| \geq C|h|^3. \]

Finally, Let $ x\in\Omega $ and set $ (\psi,t)\in D $ and $ p\in B $ such that $ p = \bfp(\psi,t) = \nabla u^\top(x) $. If $ v_\psi(\psi,t)\neq 0 $ then $ u $ is smooth near $ x $. If $ v_\psi(\psi,t) = 0 $, then for $ y $ sufficiently close to $ x $, $ q := \nabla u^\top(y) $ is close to $ p $ and
\begin{align*}
	\frac{\left|\nabla u(x) - \nabla u(y)\right|}{\left|x - y\right|^{1/3}}
	&= \frac{\left|p - q\right|}{\left|\bff(p) - \bff(q)\right|^{1/3}}\\
	&= \left(\frac{\left|p - q\right|^3}{\left|\bff(p) - \bff(q)\right|}\right)^{1/3}\\
	&\leq \frac{1}{C^{1/3}}.
\end{align*}
This completes the proof of the Theorem.

\section{Examples}

\subsection{The $ \infty $-potential in the square}

In \cite{Brustad+2026+43+59} it is shown that the (unique viscosity) solution of the Dirichlet problem
\begin{equation}\label{eq: Dirprob1}
\begin{cases}
\Delta_\infty u = 0\qquad &\text{in $\Omega\setminus\{(1,1)\}$},\\
u = 0 & \text{on $\partial\Omega$,}\\
u = 1 &\text{at $(1,1)$,}
\end{cases}
\end{equation}
in the square $\Omega = \{(x,y)\;|\; 0<x<2,\,0<y<2\}$ has an invertible gradient in the sub-square $\Omega_1 = \{(x,y)\;|\; 0<x<1,\,0<y<1\}$. Here, $ B_1 := \nabla u(\Omega_1) = \left\{(p,q)\,|\, 0<p^2 + q^2<1,\; p>0,\,q>0\right\} $ and the potential $ w(p) = p^\top\bff(p) - u(\bff(p)) $ of the inverse $ \bff\colon B_1\to\Omega_1 $ in $ (\psi, t) $-coordinates is shown to be the caloric function
\begin{equation}\label{eq:Wdef}
h(\psi, t) = \frac{8}{\pi}\left(\frac{e^{-4t}}{6}\sin(2\psi) + \frac{e^{-36t}}{210}\sin(6\psi) + \frac{e^{-100t}}{990}\sin(10\psi) + \cdots\right)
\end{equation}
in $ D = \left\{(\psi, t)\,|\, 0<\psi<\frac{\pi}{2},\; t>0\right\} $. Moreover, for $ v = -h - h_t $,
\[ v_\psi(\psi, t) = -\frac{4}{\pi}\vartheta(2\psi,e^{-16t}) \]
where $ \vartheta $ is the second Jacobi Theta function. Also, $ v_\psi $ is zero in $ D $ exactly when $ \psi = \pi/4 $, which corresponds to the diagonal in $ \Omega_1 $.

Since $ \bff $ is real-analytic it is locally Lipschitz, which in turn implies the condition \eqref{eq: reverse lipschitz} locally on $ \nabla u $. We therefore have the following Corollary to the Theorem.

\begin{corollary}
	The $ \infty $-potential in the square is $ C_{loc}^{1,\frac{1}{3}}. $
\end{corollary}

More specifically, the gradient is (at least) Lipschitz near the medians ($ x $ or $ y=1 $), it is (exactly) $ C^{1/3} $ near the diagonals, and real-analytic elsewhere.

\subsection{Aronsson's solution}
For $ u(x,y) = \frac{3}{4}\left(x^{4/3} - y^{4/3}\right) $,
\[ [p, q] = \nabla u(x,y) = [x^{1/3}, - y^{1/3}].\]
This solves to
\[ [x, y]^\top = \bff(p, q) = [p^3, -q^3]^\top, \]
which is the gradient of
\[ w(p,q) = [p,q]\bff(p,q) - u(\bff(p,q)) = p^4 - q^4 - \frac{3}{4}(p^4 - q^4) = \frac{1}{4}(p^4 - q^4). \]
Next,
\[ v(\psi, t) = u(\bff(e^{-t}\cos\psi, e^{-t}\cos\psi)) = \frac{3}{4}e^{-4t}(\cos^4\psi - \sin^4\psi) = \frac{3}{4}e^{-4t}\cos 2\psi \]
with $ v_\psi = -\frac{3}{2}e^{-4t}\sin 2\psi $, which confirms that $ u $ is non-smooth (outside the origin) if and only if the direction of the gradient is $ \{v_\psi = 0\} = \{\psi = k\pi/2\} $.

\appendix
\section{\\Proof of Lemma \ref{lem:fund}}

I explicitly state the key property needed for the proof of Lemma \ref{lem:fund}. 
\begin{lemma}\label{lem: gradient increase}
	Let $ \Omega\subseteq\mR^n $, $ n \geq 2 $, be open and assume that $ u\in C^1(\Omega) $ is $ \infty $-subharmonic in $ \Omega $. i.e., $ \Delta_\infty u\geq 0 $ in the viscosity sense.
	
	Let $ x_0\in\Omega $ and let $ 0 < r < \dist(x_0,\partial\Omega) $. If $ x\in\partial B_r(x_0) $ satisfies
	\[ u(x) = \max_{|z-x_0|\leq r}u(z) = \max_{|z-x_0|= r}u(z),\footnote{Due to the maximum principle, this second equality is always true.} \]
	then
	\[ |\nabla u(x_0)|\leq\frac{u(x) - u(x_0)}{r}\leq|\nabla u(x)|. \]
\end{lemma}

\begin{proof}
	This is Lemma 4.1(d) together with Lemma 4.3 in \cite{Crandall2008}.
\end{proof}

\begin{proof}[Proof of Lemma \ref{lem:fund}]
	I mimic the proof of Proposition 6.2 in \cite{Crandall2008}.
	The idea is to construct a sequence $ (\gamma_k)_k $ of curves $ \gamma_k : [0, T_k] \to V $, and show that a limit curve $ \gamma $ has the desired properties.
	
	For a fixed small $\delta > 0 $ let $ [x^0, x^1, \dots, x^J] $ be a list of points in $ V $ defined recursively as $ x^0 = x $, and with $ x^{j+1} $ chosen so that
	\begin{equation}\label{eq:max}
	u(x^{j+1}) = \max_{|z - x^j| = \delta}u(z).
	\end{equation}
	
	This is when $ \dist(x^j,\partial V) > \delta $. If $ \dist(x^j,\partial V) \leq \delta $ we stop and set $ J := j $ and $ T := J\delta $. We have by Lemma \ref{lem: gradient increase} that
	\begin{equation}\label{eq:gradmono}
	|\nabla u(x^j)| \leq \frac{u(x^{j+1}) - u(x^j)}{\delta} \leq |\nabla u(x^{j+1})|.
	\end{equation}
	In particular, $ |\nabla u(x^j)| > 0 $ for all $ j $ and
	\begin{equation}\label{eq:Tupperbound}
	\max_{\partial V}u \geq u(x^J) \geq u(x^{J-1}) + \delta |\nabla u(x^{J-1})| \geq\cdots\geq u(x) + J\delta |\nabla u(x)|,
	\end{equation}
	confirming that $ J $ is finite.
	Clearly, also
	\begin{equation}\label{eq:Tlowerbound}
	\dist(x,\partial V) \leq \delta + |x^J - x| \leq \delta + \sum_{j=1}^{J} |x^j - x^{j-1}| = \delta + T.
	\end{equation}
	
	Since $ u $ is differentiable, the maximum in \eqref{eq:max} implies that the gradient of $ u $ at $ x^{j+1} $ is perpendicular to the tangent space of $ \partial B_\delta(x^j) $ at $ x^{j+1} $. That is, $ \nabla u(x^{j+1}) $ is parallel to $ x^{j+1} - x^j $, and thus
	\begin{equation}\label{eq:gradeq}
	\hu^\top(x^{j+1}) = + \widehat{(x^{j+1} - x^j)} = + \frac{x^{j+1} - x^j}{\delta}
	\end{equation}
	where the "+" follows by continuity of $ \nabla u $.
	
	Now, let $ \delta_k\searrow 0 $ as $ k\to\infty $ and add a subscript $ k $ to the constructions above. For each $ k $ let $ \gamma_k : [0, T_k] \to V $ be the 1-Lipschitz piecewise linear curve connecting the points $ (x_k^j)_j $:
	\[ \gamma_k(s) = x_k^j + (s - j\delta_k)\frac{x_k^{j+1} - x_k^j}{\delta_k},\qquad j\delta_k\leq s\leq (j+1)\delta_k,\quad j=0,\dots,J_k-1. \]
	By \eqref{eq:gradeq}, and by writing $ j = \lfloor s/\delta_k\rfloor $, the curve takes the form
	\[ \gamma_k(s) = x_k^{\lfloor s/\delta_k\rfloor} + \left(s - \lfloor \frac{s}{\delta_k}\rfloor\delta_k\right)\hu^\top\left(x_k^{\lceil s/\delta_k\rceil}\right),\qquad 0\leq s\leq T_k = J_k\delta_k, \]
	with derivative
	\begin{equation}\label{eq:ae}
	\gamma_k'(s) = \hu^\top\left(x_k^{\lceil s/\delta_k\rceil}\right) = \hu^\top\left(\gamma_k(\lceil s/\delta_k\rceil\delta_k)\right)
	\end{equation}
	almost everywhere on $ [0, T_k] $.
	
	Set $ T := \liminf_{k\to\infty}T_k $. For $ \epsilon>0 $ we have $ T_k \geq T-\epsilon $ for all $ k $ greater than some $ K_\epsilon $.
	Arzelà-Ascoli ensures a 1-Lipschitz curve $ \gamma : [0,T-\epsilon]\to V $ and a subsequence $ (\gamma_{k_i})_i $, $ k_i>K_\epsilon $, such that $ \gamma_{k_i} $ converges uniformly to $ \gamma $ on $ [0,T-\epsilon] $. But since also $ \nabla u $ is uniformly continuous in $ \overline{V} $ and $ \lceil s/\delta_k\rceil\delta_{k_i} $ goes uniformly to $ s $, it follows that the right-hand side of \eqref{eq:ae} converges uniformly to $ \hu^\top\left(\gamma(s)\right) $ on $ [0, T-\epsilon] $. The limit can then be taken outside an integral, and thus
	\begin{align*}
	\int_0^s\hu^\top\left(\gamma(\sigma)\right)\dd\sigma
	&= \lim_{i\to\infty} \int_0^s\hu^\top\left(\gamma_{k_i}(\lceil \sigma/\delta_{k_i}\rceil\delta_{k_i})\right)\dd\sigma\\
	&= \lim_{i\to\infty} \int_0^s \gamma_{k_i}'(\sigma)\dd\sigma\\
	&= \lim_{i\to\infty} \gamma_{k_i}(s) - \gamma_{k_i}(0)\\
	&= \gamma(s) - x
	\end{align*}
	for all $ s\in [0, T-\epsilon] $. It follows that $\gamma$ is differentiable with the continuous derivative $\gamma'(s) = \hu^\top\left(\gamma(s)\right)$.
	
	By the Lipschitz continuity and the definition of $ T_k $,
	\[ \dist(\gamma_{k_i}(T-\epsilon), \partial V) \leq \dist(\gamma_{k_i}(T_{k_i}), \partial V) + \epsilon \leq \delta_{k_i} + \epsilon. \]
	Decreasing $ \epsilon $ thus extends $ \gamma(T-\epsilon) $ towards $ \partial V $, and in the limit $ \gamma(T)\in\partial V $.
	
	Finally, part (iii) follows by \eqref{eq:Tupperbound} and \eqref{eq:Tlowerbound}, (v) follows by \eqref{eq:gradmono}, and $ s\mapsto u(\gamma(s)) $ is convex since then
	\[ \frac{\dd}{\dd s} u(\gamma(s)) = \nabla u(\gamma(s))\gamma'(s) = |\nabla u(\gamma(s))| \]
	is non-decreasing.
	
\end{proof}


\bibliographystyle{alpha}
\bibliography{/Users/karlkb/Documents/references.bib}


\end{document}